\documentclass{article}
\usepackage[utf8]{inputenc}
\usepackage{amsmath,amsthm,amssymb}
\usepackage{url}
\usepackage{graphicx}
\usepackage{float}

\renewcommand{\H}{\mathbb{H}}
\newcommand{\R}{\mathbb{R}}
\newcommand{\C}{\mathbb{C}}
\newcommand{\D}{\mathbb{D}}
\newcommand{\isom}{\text{Isom}^+}
\newcommand{\SO}{\text{SO}(3)}

\newtheorem{theorem}{Theorem}
\newtheorem{question}[theorem]{Question}

\newtheorem{proposition}[theorem]{Proposition}
\newtheorem{lemma}[theorem]{Lemma}

\author{Ernesto García \and Pablo Lessa}
\title{On the discretness of states accessible via right-angled paths in hyperbolic space}
\DeclareMathOperator{\acosh}{acosh}
\DeclareMathOperator{\dist}{dist}

\let\sl=\newsl
\DeclareMathOperator{\psl}{PSL}
\DeclareMathOperator{\tr}{tr}
\DeclareMathOperator{\im}{Im}

\begin{document}
 \maketitle
 
 \begin{abstract}
  We consider the control problem where, given an orthonormal tangent frame in the hyperbolic plane or three dimensional hyperbolic space, one is allowed to transport the frame a fixed distance \(r > 0\) along the geodesic in direction of the first vector, or rotate it in place a right angle.  We characterize the values of \(r > 0\) for which the set of orthonormal frames accessible using these transformations is discrete.
  
  In the hyperbolic plane this is equivalent to solving the discreteness problem (see \cite{gilman2} and the references therein) for a particular one parameter family of two-generator subgroups of \(\text{PSL}_2(\R)\).   In the three dimensional case we solve this problem for a particular one parameter family of subgroups of the isometry group which have four generators.
 \end{abstract}

 \section{Introduction}
 
 Imagine a robot which can move forward a fixed distance and rotate in place a right angle.  Which states are accessible for such a system from a given initial position and orientation?
 
 It is clear that, if placed on the Euclidean plane, the robot is constrained to move on a square grid.  The attainable states (position and orientation) for the robot are the vertices of the grid and the four orientations parallel to the edges.
 
 However, if we imagine the robot constrained to move on the surface a sphere, it is simple to see that the accessible states may form either a finite or infinite set, 
 depending on the relationship between the distance the robot is allowed to advance and the diameter of the sphere.

 To formalize this consider the unit sphere \(S^2\) centered at the origin in \(\R^3\). Fixing any point \(p \in S^2\) and unit tangent vector \(v\) at \(p\) to represent the initial state of the robot,  we may identify the space of possible states with the group of rotations \(\SO\) via the mapping \(R \mapsto (Rp,Rv)\).  
 
 With this identification the set of accessible states corresponds to the subgroup \(G_r\) of \(\SO\) generated by the two elements \(R\) and \(A_r\), where \(R\) is the unique rotation fixing \(p\) and rotating \(v\) a clockwise right angle, and \(A_r\) is the unique rotation advancing \((p,v)\) a distance \(r\) along the geodesic (great circle) with initial speed \(v\).
 
 In particular, the set of accessible states is finite if and only if \(G_r\) is a finite subgroup of \(\SO\).    
 
 The finite subgroups of \(\SO\) are well known (see for example \cite[Chapter 19]{armstrong}) and their classification implies that \(G_r\) is discrete if and only if \(r = n\pi/2\) for some integer \(n\). 
 
 In what follows we answer the above question in the hyperbolic plane \(\H^2\), and also in three dimensional hyperbolic space \(\H^3\).

\subsection{Statements}

Fix an orientation on \(\H^2\) and let \(\isom(\H^2)\) be its group of orientation preserving isometries.  As in the spherical case we fix an initial point \(p \in \H^2\) and a unit tangent vector \(v \in T_p\H^2\) which represent the initial position and orientation of the robot.   The set of states is identified with the group \(\isom(\H^2)\) via the mapping \(g \mapsto (g(p), D_pg (v))\).

Let \(R\) be the clockwise rotation by a right angle fixing \(p\), and \(A_r\) be the translation of distance \(r\) along the positive direction of the geodesic with initial condition \(v\).
 
The set of accessible states is the orbit of \(v\) under the group \(G_r\) generated by \(R\) and \(A_r\).  For \(r > 0\) this set is always infinite, but we are interested in whether it is discrete or not.

We will use \(\acosh(x) = \log(x + \sqrt{x^2 -1})\) to denote the inverse hyperbolic cosine function.

Our main result is the following:

\begin{theorem}\label{maintheorem}
  Let \(r_5 < r_6 < \cdots\) be the sequence where \(r_n = \acosh(1 + 2\cos(\frac{2\pi}{n}))\) is the side of the (unique up to isometry) regular \(n\)-gon with interior right angles in \(\H^2\), let \(r_\infty = \lim\limits_{n \to +\infty} r_n = \acosh(3)\), and let \(G_r\) be the group generated by \(R\) and \(A_r\) (as defined above).

  Then \(G_r\) is discrete if and only if \(r \in \lbrace r_n: n \ge 5\rbrace \cup [r_\infty, +\infty)\).
  For all other values of \(r\) the group \(G_r\) is dense in \(\isom(\H^2)\).
\end{theorem}

Using the Poincaré Polygon Theorem we will show in section \ref{tilingssection} that for \(n \ge 5\) the group \(G_{r_n}\) acts with a fundamental domain given by a triangle with angles \(\pi/4,\pi/4,2\pi/n\).

In section \ref{sectiontrees} we will show that when \(r \ge r_\infty\) the group \(G_{r}\) preserves a embedded tree of degree 4, which in particular shows that \(G_{r}\) is discrete.

In sections \ref{sectionirrational} and \ref{sectionnonprimitive} we show that in the remaining cases \(G_r\) is not discrete, using Jørgensen's inequality.  This implies that \(G_r\) is dense in these cases by a well known dichotomy (see Proposition \ref{dichotomy}, and \cite[Section 1]{sullivan}).

We will extend theorem \ref{maintheorem} to three dimensional hyperbolic space \(\H^3\) as follows.

Let \(\isom(\H^3)\) be the group of orientation preserving isometries of \(\H^3\).   Fix a point \(p \in \H^3\) and an orthogonal tangent frame \(v_1,v_2,v_3\) based at \(p\).
Suppose \(A_r\) is the isometry which transports the given frame a distance \(r\) along the geodesic 
with initial condition \((p,v_1)\) while \(R_{12},R_{23}, R_{31}\) are \(90º\) rotations fixing \(p\) in the direction of the planes generated by \((v_1,v_2),(v_2,v_3)\), and \((v_3,v_1)\) respectively.

\begin{theorem}\label{dimensionthreetheorem}
 The subgroup \(G_r\) of \(\isom(\H^3)\) generated by \(A_r,R_{12},R_{23},R_{31}\) is discrete if and only if \(r \in \lbrace r_n: n \ge 5\rbrace \cup [r_\infty, +\infty)\) where \(r_n\) are defined as in Theorem \ref{maintheorem}.   For all other values of \(r\) the group is dense in \(\isom(\H^3)\).

 Furthermore, \(G_{r_5}\) is cocompact, \(G_{r_6}\) is not cocompact but has finite covolume, and \(G_{r_n}\) has infinite covolume for all \(n \ge 7\).
 \end{theorem}

The proof of theorem \ref{dimensionthreetheorem} rests on theorem \ref{maintheorem}, Andreev's theorem (see \cite{roeder-hubbard-dunbar}), and the Poincaré Polyhedron Theorem for reflexion groups (see \cite{delaharpe}).  It is given in section \ref{threedimensionalsection}.

\subsection{Relationship to the existing literature}

The discreteness problem is the problem of determining whether a finitely generated group of isometries of hyperbolic space is discrete (see \cite{gilman2} and the references therein).  
Both of our results are solutions to this problem for particular one parameter families of subgroups of isometries in \(\H^2\) and \(\H^3\) respectively.

The Gilman-Maskit algorithm (see \cite{gilman} and \cite{gilman-maskit}) gives a finite sequence of steps to determine whether a two generator subgroup of \(\isom(\H^2)\) is discrete or not.  
Theorem \ref{maintheorem} gives additional information on the structure of the set of parameters for which the algorithm will yield one result or the other.  An illustration of the arguments of \cite{gilman} applied to one parameter considered here is given in section \ref{sectiongilman}.

It was shown in \cite{kapovich} that no real number algorithm exists for determining whether a finitely generated group of isometries of \(\H^3\) is discrete or not.

The results of \cite{gruet} and \cite{cohen-colonna} are related to a generalization of the family of groups \(G_r\) where the rotation \(R\) has order \(2N\) for some \(N \ge 2\) (instead of \(4\) as in theorem \ref{maintheorem}).  For further discussion see section \ref{sectiongruet}.

In the context of theorem \ref{maintheorem}, the Poincaré polygon theorem (see for example \cite{derham} or \cite{maskit}) directly implies that \(G_{r_n}\) is discrete for \(n = 5,6,\ldots\) and preserves a tiling by regular polygons.  

The Poincaré theorem also holds in higher dimensional hyperbolic space (see \cite{epstein}) and is sufficient to establish that \(G_{r_5}\) is discrete and cocompact in the context of theorem \ref{dimensionthreetheorem}.   However a simple argument comunicated to us by Roland Roeder, which we give in section \ref{threedimensionalsection} shows that the corresponding compact polyhedra for \(r_6,r_7,\ldots\) do not exist in \(\H^3\).  
We will show, however, that an infinite volume polyhedra corresponding to each \(r_n\) for \(n=6,7,\ldots\) does exist, and apply the Poincaré polyhedron theorem (specifically the version for reflexion groups given here \cite{delaharpe}) to it to obtain theorem \ref{dimensionthreetheorem}.

In the context of both theorem \ref{maintheorem} and theorem \ref{dimensionthreetheorem}, when \(r > r_\infty\) the group \(G_r\) contains a finite index Schottky group and therefore its behavior is well understood in the literature.  For example, in the two dimensional case, the results of \cite{cohen-colonna} imply that \(G_r\) leaves invariant a regular tree of degree \(N\) which is bi-Lipschitz embedded in the hyperbolic plane.
This case is also covered by results on finite valued matrix cocycles in \cite{avila-bochi-yoccoz}.
At the critical distance \(r_\infty\) there is still an embedded tree preserved by \(G_{r_\infty}\) but the embedding is no longer bi-Lipschitz.

A well known argument (see for example \cite[Part 1]{sullivan}) implies that in both the two and three dimensional cases, for each \(r > 0\) either \(G_r\) is discrete or dense in the corresponding group \(\isom(\H^2)\) or \(\isom(\H^3)\).  We give the details of this argument for \(\H^2\) in Proposition \ref{dichotomy}.   This implies the dichotomy for \(\H^3\), as we explain at the beginning of section \ref{threedimensionalsection}.

The Margulis lemma (see for example \cite[Theorem 9.5]{ballmann-gromov-schroeder}) implies that \(G_r\) is dense for all \(r\) small enough.

Results of Benoist and Quint imply a discretness vs denseness dichotomy for \(G_r\) when acting on any finite area quotient of \(\H^2\) as discussed in \cite{ledrappier}.

Software implementations of a `robot' (usually called a `turtle' in this context) receiving commands to move forward or turn in place by given amounts date back to the LOGO programming language \cite{abelson-disessa}.
Some implementation details and exploration of the hyperbolic case is given in \cite{simscoomber-martin-thorne}.   A rudimentary but functional software implementation of a hyperbolic turtle has been made available by one of the authors \cite{lessa}.  Several of the figures in this article were prepared with the software available here \cite{lessa2}.

\section{Preliminaries}

We now recall some basic facts on hyperbolic geometry which will be used in what follows, see \cite{beardon} for a general reference on this subject.

The hyperbolic plane \(\H^2\) is the unique, up to isometry, complete simply connected surface with curvature \(-1\).   Concrete manifolds with explicit metrics satisfying these properties are called models of the hyperbolic plane.

The upper half-plane model is the space \(\lbrace z \in \C: \im(z) > 0\rbrace\) endowed with the Riemannian metric \(\frac{1}{y^2}(dx^2 + dy^2)\).   The orientation preserving isometries in this model are the Möebius transformations of the form
\[z \mapsto \frac{az + b}{cz + d}\]
where \(a,b,c,d \in \R\) and \(ad - bc = 1\).

The disk model is the space \(\D = \lbrace z \in \C: |z| < 1\rbrace\) with the metric \(\left(\frac{2}{1-(x^2+y^2)}\right)^2(dx^2+ dy^2)\).    Figures \ref{figuretiling}, \ref{figuretree}, \ref{figureirrational}, \ref{figurenonprimitive}, and \ref{figuregilman} below illustrate the disk model.  An isometry between the upper half-plane and disk model is  \(z \mapsto \frac{z-i}{z+i}\).

In both of these models the hyperbolic geodesics are Euclidean straight lines or circles which are perpendicular to the boundary.   In particular there is a unique, globally minimizing, geodesic between any pair of points in \(\H\).

An orientation preserving isometry of \(\H^2\) is called elliptic, parabolic, or hyperbolic, acording to whether it fixes a single interior point, a single boundary point, or two boundary points, respectively, in the disk model.

Elliptic isometries are also called rotations, they act as rotations in the tangent space of their fixed point in \(\H^2\).   An element
\[\pm \begin{pmatrix}\cos(\theta) & -\sin(\theta)\\ \sin(\theta) & \cos(\theta)\end{pmatrix} \in \psl_2(\R)\]
acts as a rotation of angle \(-2\theta\) (i.e. a clockwise rotation) in the half-plane model.

Hyperbolic isometries are also called translations, they fix a unique geodesic in \(\H^2\) and act as a translation of a certain distance when restricted to this geodesic.  An element
\[\pm \begin{pmatrix}e^{t} & 0\\ 0 & e^{-t}\end{pmatrix} \in \psl_2(\R)\]
acts as a translation of distance \(2t\) in the upper half-plane model.

From the Gauss-Bonnet theorem and explicit construction in one of the models shows that there exists a geodesic triangle in \(\H^2\) with interior angles \(\alpha,\beta,\gamma\) in \(\H^2\) if and only if \(\alpha+\beta+\gamma < \pi\), and in this case the triangle is unique up to isometries.

In the hyperbolic plane given the length of two sides of a triangle and the angle between them the length of the third side is determined by the hyperbolic law of cosines
\[\cosh(c) = \cosh(a)\cosh(b) - \sinh(a)\sinh(b)\cos(\gamma)\]
where \(a,b,c\) are the lengths of the sides opposite to angles \(\alpha,\beta,\gamma\) respectively.

As in spherical geometry, in \(\H^2\) two angles of a triangle and the length of the side between determine the third angle (in Euclidean geometry the length of the side plays no role in this relation).  In the hyperbolic case the relation is given by the second hyperbolic law of cosines which states
\[\cos(\gamma) = -\cos(\alpha)\cos(\beta) + \sin(\alpha)\sin(\beta)\cosh(c).\]

Let \(\isom(\H^2)\) be the group of orientation preserving isometries of \(\H^2\) endowed with the topology of pointwise convergence (which in this case is equivalent to locally uniform convergence because all functions are uniformly Lipschitz).   The upper half-plane model shows that \(\isom(\H^2)\) is homeomorphic to \(\psl_2(\R)\) with the topology of pointwise convergence coming from \(\sl_2(\R)\).

A Fuchsian group is a discrete subgroup of \(\isom(\H^2)\) (i.e. a subgroup which is discrete as a subset with respect to the given topology).

Given a subgroup \(G\) of \(\isom(\H^2)\) if the orbit \(Gp\) is not discrete for some \(p \in \H^2\) then \(G\) is not discrete.   On the other hand if \(G\) has a finite index subgroup \(H\) which is discrete it follows that \(G\) is discrete.

From the map \(z \mapsto \overline{z}\) in the disk model, one obtains that given a geodesic in \(\H^2\) there is a unique orientation reversing isometry that acts as the identity on the geodesic.   We call this the axial symmetry with respect to the geodesic.

If \(\sigma_1,\sigma_2\) are axial symmetries along two geodesics then their composition \(\sigma_1\sigma_2\) yields, a rotation of angle \(2\theta\) if the geodesics meet at an angle \(\theta\), a parabolic isometry if the geodesics do not intersect but the distance between them is zero, and a translation of distance \(2t\) if the geodesics are at a positive distance \(t\).

\section{Proof of theorem \ref{maintheorem}}

We fix in this section the notation introduced preceeding the statement of theorem \ref{maintheorem}.  In particular, recall that we have fixed a point \(p \in \H^2\) and a unit tangent vector \(v\) based at \(p\).   We let \(A_r\) be the translation of distance \(r\) in direction \(v\) along the geodesic with initial condition \(v\) and \(R\) be the clockwise rotation by \(90º\) fixing \(p\).    The group \(G_r\) is generated by \(A_r\) and \(R\).

\subsection{Tilings\label{tilingssection}}

We will now discuss the values of \(r\) for which \(G_r\) is discrete and preserves a tiling by regular polygons.  Our result will be a consequence of the Poincaré Polygon Theorem (see \cite{derham} or \cite{maskit}) which we now state in a version sufficient for our purpose:
\begin{theorem}[Poincaré Polygon Theorem]\label{poincare}
Suppose that \(P\) is a compact polygon in \(\H^2\) with an even number \(2N\) of sides which are oriented so that each vertex of \(P\) is the endpoint and starting point of some edge.

Divide the edges of \(P\) into \(N\) pairs \((s_1,t_1), \ldots, (s_N,t_N)\).  Suppose that for each pair of sides \((s_i,t_i)\) an orientation preserving isometry \(\sigma_i\) is given such that the interior of \(\sigma_i(P)\) is disjoint from \(P\) and such that \(\sigma_i(s_i) = t_i\).

If a vertex \(p\) is the starting point of an edge \(s_i\) we define \(\sigma_i(p)\) as its successor, if on the other hand \(p\) is the starting point of an edge \(t_i\) we define \(\sigma_i^{-1}(p)\) as its successor.  An elliptic cycle is the complete orbit of a vertex under the successor mapping.

If the sum of interior angles among the vertices of each elliptic cycle is \(2\pi/k\) for some natural number \(k\) (depending on the cycle) then the group generated by \(\sigma_1,\ldots,\sigma_N\) is discrete, the translates of \(P\) under this group cover \(H\), and no two translates of \(P\) by distinct elements of the group intersect at an interior point.
\end{theorem}

The conclusions of the above theorem can be restated by saying that \(P\) is a fundamental domain for the group generated by \(\sigma_1,\ldots, \sigma_N\).

A geodesic triangle with interior angles \(\pi/4, \pi/4\) and \(2\pi/n\) (where \(n\) is a natural number) exists in \(\H^2\) if and only if
\[\frac{\pi}{4} + \frac{\pi}{4} + \frac{2\pi}{n} < \pi\]
which implies \(n \ge 5\).

We consider such a triangle \(T_n\) with a vertex at \(p\) and edges in directions \(v\) and \(w\) forming a clockwise angle of \(\pi/4\), and such that the edge in direction \(v\) is opposite to the interior angle \(2\pi/n\).

By the second hyperbolic law of cosines, the length of the side with direction \(v\) has length
\[r_n = \acosh\left(1+ 2\cos\left(\frac{2\pi}{n}\right)\right).\]

Notice that the sequence \(r_n\) is increasing, we define 
\[r_\infty = \lim\limits_{n \to +\infty} r_n = \acosh(3).\]

\begin{proposition}\label{polygons}
For each \(n = 5,6,\ldots\) the group \(G_{r_n}\) is discrete and \(T_n\) is a fundamental domain.
\end{proposition}
\begin{proof}
 Let \(m\) be the midpoint of the geodesic segment \([p,A_{r_n}(p)]\), and let \(a\) and \(b\) be the other two sides of \(T_n\) where \(a\) has and endpoint at \(p\).
 
 The isometry \(\sigma_1 = A_{r_n}R^2\) fixes \(m\) and sends the geodesic segment \([p,m]\) to \([m,A_{r_n}(p)]\).    The isometry \(\sigma_2 = A_{r_n}R\) maps \(a\) to \(b\) fixing their shared endpoint.
 
 By theorem \ref{poincare} the group generated by \(\sigma_1\) and \(\sigma_2\) is discrete and has \(T_n\) as a fundamental domain.   Since this group coincides with \(G_{r_n}\) this establishes the claim.
\end{proof}

\begin{figure}
\includegraphics[width=\textwidth]{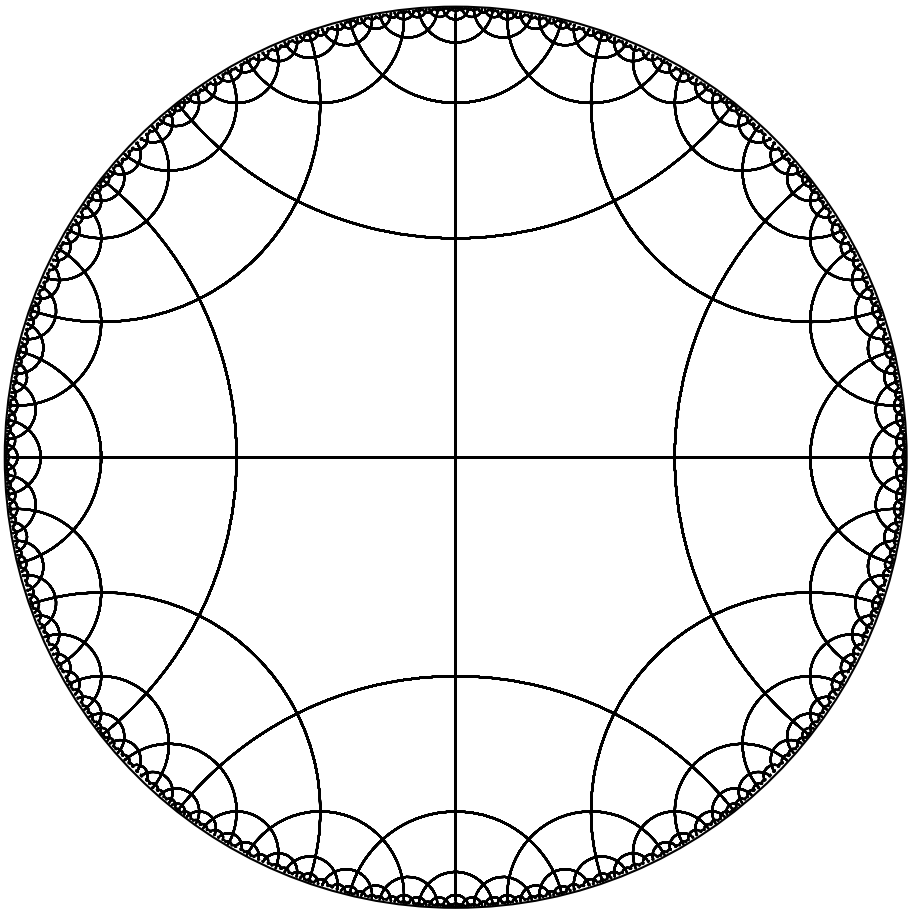}
\caption{The tiling corresponding to \(r = r_5\).\label{figuretiling}}
\end{figure}

\subsection{Trees\label{sectiontrees}}

We will now discuss the case where the group \(G_r\) is discrete and preserves and embedded regular tree of degree four.  This happens exactly when \(r \ge r_\infty\).

\begin{figure}
\includegraphics[width = \textwidth]{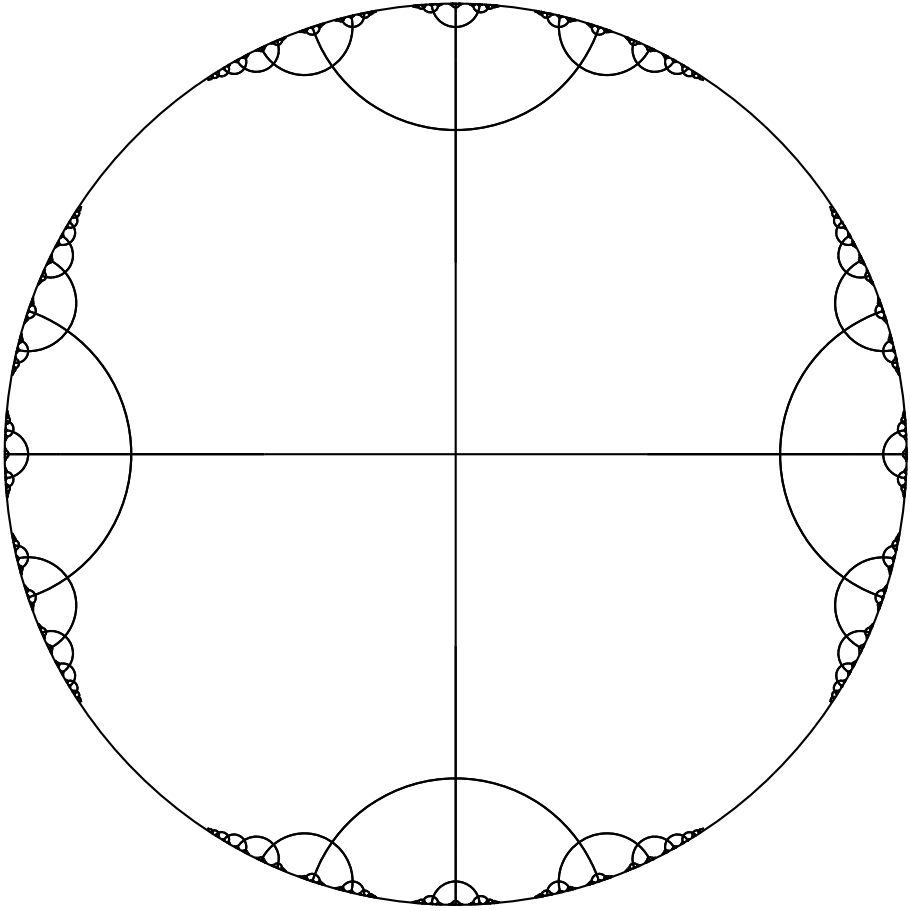}
\caption{The embedded tree for \(r = r_\infty + 0.05\).\label{figuretree}}
\end{figure}

For this purpose let \(B_r = RA_r R^{-1}\) and \(H_r\) be the group generated by \(A_r\) and \(B_r\).

Also, we define the four closed half-planes \(N,S,E,W\) (for North, South, East, and West respectively) by
\[N = \lbrace q:  \dist(q,p) \ge \dist(q,B_r^{-1}(p))\rbrace,\]
\[S = \lbrace q:  \dist(q,p) \ge \dist(q,B_r(p))\rbrace,\]
\[E = \lbrace q:  \dist(q,p) \ge \dist(q,A_r(p))\rbrace,\]
\[W = \lbrace q:  \dist(q,p) \ge \dist(q,A_r^{-1}(p))\rbrace,\]
where \(\dist(a,b)\) is the hyperbolic distance between \(a\) and \(b\).

We define the central region \(C = \H \setminus (N \cup S \cup E \cup W)\).

\begin{proposition}\label{trees}
 The regions \(N,S,E,W\) are pairwise disjoint if and only if \(r \ge r_\infty\).   
 In this case \(H_r\) is discrete and freely generated by \(A_r\) and \(B_r\).   Also, \(G_r\) is discrete, the stabilizer of \(C\) in \(G_r\) is generated by \(R\), and \(G_r\) preserves a geodesic embedding of the regular tree of degree four.
\end{proposition}
\begin{proof}
 Without loss of generality assume the regions \(N\) and \(E\) intersect.  Let \(x,y\) be the closest points to \(p\) in \(N\) and \(E\) respectively.  Observe that the geodesic segments \([p,x]\) and \([p,y]\) have length \(r/2\) and meet at a right angle at \(p\).
 
 If \(N \cap E \neq \emptyset\) then there is an intersection point \(z\) which is in the boundary of both regions and is equidistant from \(x\) and \(y\).
 
 Consider the triangle \(T\) with vertices \(p,x,z\) and let \(\alpha\) be the interior angle at \(z\).  Notice that the interior angle at \(x\) is \(\pi/2\) and at \(p\) is \(\pi/4\) so that \(0 < \alpha < \pi/4\).   By the second hyperbolic law of cosines one has
 \begin{equation}\label{leycosenosejes}r = 2\acosh(\sqrt{2}\cos(\alpha)) < 2\acosh(\sqrt{2}) = \acosh(3) = r_\infty.\end{equation}
 
 Conversely, if the inequality above is satisfied a triangle with angles \(\pi/2,\pi/4,\alpha\) exists in \(\H^2\).  Placing two such triangles with right angles at \(x\) and \(y\) respectively, sharing a vertex at \(p\) and a side along the perpendicular bisector of \(x\) and \(y\) it follows that \(N\) and \(E\) intersect at a common third vertex \(z\).
 
 Suppose now that \(r \ge r_\infty\) so that \(N,S,E,W\) are pairwise disjoint.
 
 Notice that \(A_r(\H^2 \setminus W) \subset E\) ,\( A_r^{-1}(\H^2 \setminus E) \subset W\), \(B_r(\H^2 \setminus N) \subset S\), and \(B_r^{-1}(\H^2 \setminus S) \subset N\).
 
 This shows that if \(X\) is any non-trivial reduced word in \(A_r, A_r^{-1}, B_r, B_r^{-1}\) (i.e. a finite product where no element is followed by its inverse) then \(X(C) \cap C = \emptyset\).  Hence the group \(H_r\) is freely generated by \(A_r\) and \(B_r\) (this is an instance of the well known ping-pong lemma, see for example \cite{koberda2012}) and is discrete since the orbit of \(p\), and the stabilizer of \(p\) is trivial.

 Notice that \(G_r\) is generated by \(H_r\) and \(R\), and \(RH_rR^{-1} = H_r\) (it suffices to check \(RA_rR^{-1} \in H_r\) and \(RB_rR^{-1} \in H_r\)).  Hence, \(H_r\) is a normal subgroup of \(G_r\).
 
 Furthermore, since \(R(C) = C\) it follows that if \(X,Y \in H_r\) and \(XR^i = YR^j\) then \(X = Y\) and \(i = j\ (\text{mod}\ 4)\).  This shows that \(H_r, H_r R, H_r R^2\) and \(H_r R^3\) are pairwise disjoint.  It follows that their union must be \(G_r\) and \([G_r:H_r] = 4\).  If \(XR^i\) is an element of \(G_r\) we have that \(XR^i(C) = X(C) = C\) if and only if \(X\) is the identity, so the stabilizer of  \(C\) in \(G_r\) is generated by \(R\) as claimed.
 
 Let \(x,y,z,w\) be the closest points to \(p\) in the regions \(N,S,E,W\) respectively.  The set \(\lbrace p,x,y,z,w\rbrace\) is \(R\)-invariant and the union of geodesic segments \(S = [p,x] \cup [p,y] \cup [p,z] \cup [p,w]\) intersect pairwise only at \(p\).   Notice that \(X(S) \cap S \neq \emptyset\) for a non-identity element \(X \in H_r\) if and only if \(X \in \lbrace A_r, A_r^{-1},B_r,B_r^{-1}\rbrace\) in which case the intersection is a single point from the set \(\lbrace x,y,z,w\rbrace\).  It follows that the \(H_r\)-orbit of \(S\) is a tree of degree four (the Cayley graph of \(H_r\)) with all edges of length \(r\), and is invariant under \(G_r\) as claimed.
\end{proof}

\subsection{Irrational rotations\label{sectionirrational}}

In this section and the following one we will show that the only values of \(r > 0\) for which \(G_r\) is discrete are given by propositions \ref{polygons} and \ref{trees}.   We will also show that for all other values of \(r\) the group \(G_r\) is dense in \(\isom(\H^2)\).

\begin{figure}
\begin{minipage}{\textwidth}
\begin{minipage}{0.32\textwidth}
\includegraphics[width=\textwidth]{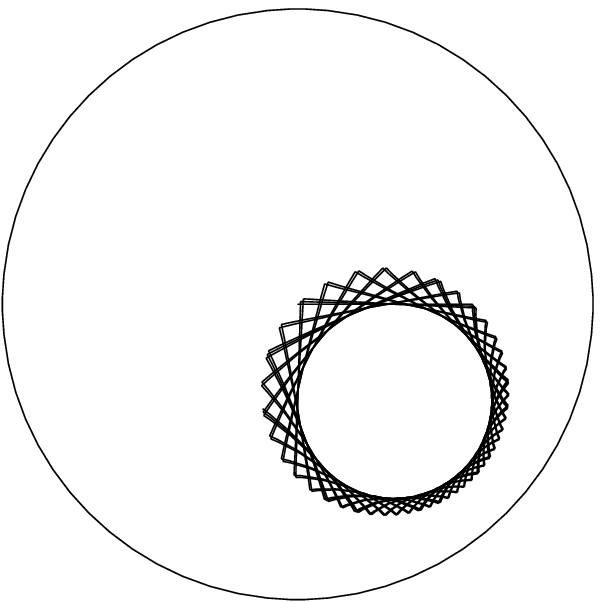}
\end{minipage}
\begin{minipage}{0.32\textwidth}
\includegraphics[width=\textwidth]{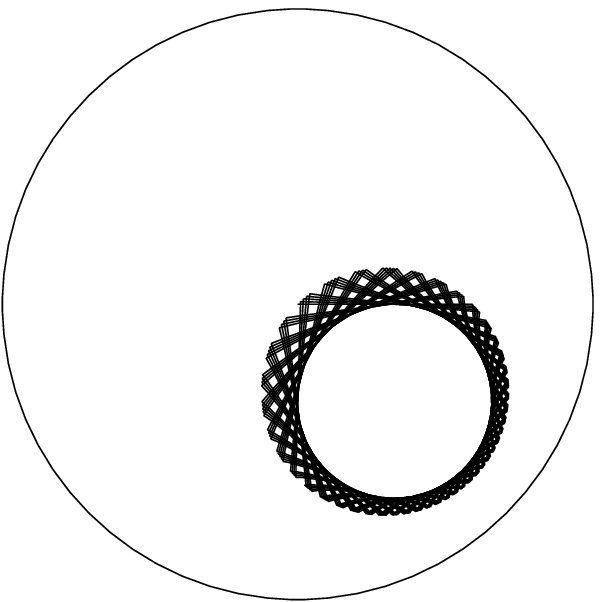}
\end{minipage}
\begin{minipage}{0.32\textwidth}
\includegraphics[width=\textwidth]{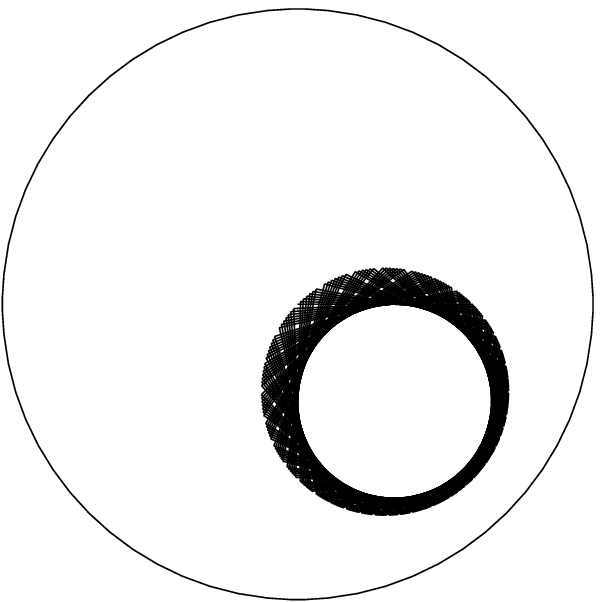}
\end{minipage}
\caption{An illustration of \(100,200\) and \(300\) iterations of \(A_rR\) applied to a segment of length \(r\) for \(r = r_{2\pi}\).\label{figureirrational}}
\end{minipage}
\end{figure}

For this purpose the first important observation is the following:
\begin{proposition}\label{rinftyelliptic}
The isometry  \(A_r R\) is elliptic if and only if \(r < r_\infty\).
\end{proposition}
\begin{proof}
We repeat the argument from the proof of Proposition \ref{trees}.

Let \(\sigma_2\) be the axial symmetry (orientation reversing isometry which is the identity along a geodesic) with respect to the geodesic passing through \(p\) in direction perpendicular to \(v\).

Define \(\sigma_1 = A_{r/2}\sigma_2 A_{r/2}^{-1}\) and notice that \(A_r = \sigma_1\sigma_2\).

Letting \(\sigma_3\) be the symmetry with respect to the geodesic passing through \(p\) in a direction \(45º\) clockwise from \(v\), notice that \(R = \sigma_2\sigma_3\).

To conclude observe that \(AR = \sigma_1\sigma_2\sigma_2\sigma_3 = \sigma_1\sigma_3\) is elliptic if and only if the geodesics fixed by \(\sigma_1\) and \(\sigma_3\) intersect.   If this happens there exist a triangle with a side of length \(r/2\) adjacent to angles \(\pi/2\) and \(\pi/4\).   By the second law of cosines (see equation \ref{leycosenosejes}) this happens if and only if \(r < r_\infty\).
\end{proof}

The following well known argument (see \cite[Section 1]{sullivan}) shows that if \(G_r\) is not discrete then it is dense in \(\isom(\H^2)\).  In particular, this happens if \(A_r R\) is elliptic of infinite order.

\begin{proposition}\label{dichotomy}
For each \(r > 0\) either \(G_r\) is discrete or dense in \(\isom(\H^2)\).
\end{proposition}
\begin{proof}
 We use the Poincaré model where \(\H^2\) is identified with the unit disc \(\D = \lbrace z \in \C: |z| < 1\rbrace\) and \(\isom(\H^2)\) with the group \(M\) of complex Möbius transformations preserving \(\D\).
 
 Let \(S\) be the closure of \(G_r\) in \(M\) and \(S_0\) the connected component of the identity in \(S\).  Notice that \(S_0\) is normal in \(S\) and is a connected Lie subgroup of \(M\).
 
 If \(S_0\) has dimension \(0\) then \(S\) (and therefore \(G_r\)) is discrete.  We suppose from now on that this is not the case. 
 
 If \(S_0\) is a proper subgroup of \(M\) then there is a non-empty set \(F\) with at most two points in the closed disk \(\overline{\D}\) such that \(S_0\) is the set of elements in \(M\) fixing all points in \(F\).
 
 Since \(S_0\) is normal in \(S\) it follows that all elements of \(S\) permute the points in \(F\).
 
 However, it is immediate that no finite set in \(\overline{\D}\) is invariant by both \(A_r\) and \(B_r = RA_rR^{-1}\).   Therefore \(S_0 = M\) and \(G_r\) is dense as claimed. 
\end{proof}

\subsection{Non-primitive rotations\label{sectionnonprimitive}}

\begin{figure}
\includegraphics[width = \textwidth]{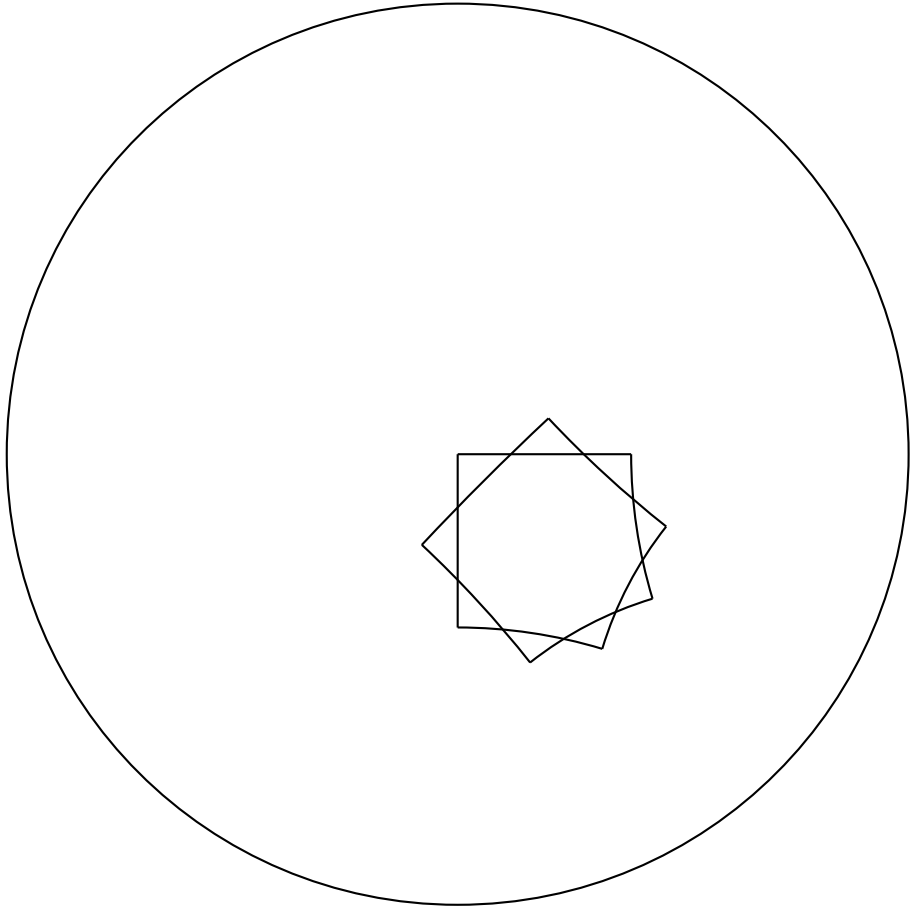}
\caption{A non-simple right angled polygon with sides of length \(r = r_{9/2}\).\label{figurenonprimitive}}
\end{figure}

We extend the definition of \(r_n\) used in proposition \ref{polygons} to all \(t > 4\) with the formula
\[r_t = \acosh\left(1+ 2\cos\left(\frac{2\pi}{t}\right)\right).\]

It is simple to see that \(t \mapsto r_t\) is an increasing homeomorphism from \((4,+\infty)\) to \((0,r_\infty)\) and that \(A_{r_t}R\) is a rotation of angle \(2\pi/t\).

If \(t\) is irrational then by proposition \ref{dichotomy} the group \(G_{r_t}\) is dense in \(\isom(\H^2)\).  Proposition \ref{polygons} shows that if \(t = 5,6,7,\ldots\) then \(G_{r_t}\) is discrete.

It remains to discuss the case \(t = p/q > 4\) where \(p\) and \(q\) are coprime and \(q > 1\).   We will show that \(G_{r_t}\) is dense for these values of \(t\).

For this purpose we will use Jørgensen's inequality (see \cite{jorgensen}) applied to well chosen elements of \(G_r\).

\begin{theorem}[Jørgensen's inequality]
Let \(G\) be a non-elementary Fuchsian group generated by two elements \(X,Y \in \psl_2(\R)\), then
\[|\tr(X)^2 - 4| + |\tr([X,Y]) - 2| \ge 1,\]
where \(\tr(Z)\) denotes the trace of a matrix \(Z\) and \([X,Y] = XYX^{-1}Y^{-1}\) is the commutator of \(X\) and \(Y\).
\end{theorem}

\begin{proposition}\label{nonprimitiveproposition}
If \(t = p/q > 4\) is a reduced fraction with \(q > 1\) then \(G_{r_t}\) is dense in \(\isom(\H^2)\).
\end{proposition}
\begin{proof}
We use the upper half plane model where \(\H^2\) is identified with \(\lbrace z \in \C: \im(z) > 0\rbrace\).   
We fix \(p = i\) and \(v = i\) (the unit tangent vector pointing upwards with base point \(i\)).
The group \(\isom(\H^2)\) is identified with \(\psl_2(\R)\) where 
\[\begin{pmatrix}a & b\\ c & d\end{pmatrix}\]
corresponds to the isometry \(z \mapsto \frac{az + b}{cz + d}\).

With this identification we have
\[A_r = \begin{pmatrix}e^{r/2} & 0\\ 0 & e^{-r/2}\end{pmatrix}, R = \frac{1}{\sqrt{2}}\begin{pmatrix}1 & -1\\ 1 & 1\end{pmatrix}.\]

In general if \(R_\theta\) is the matrix corresponding to the clockwise rotation of angle \(\theta\) fixing \(p\) one has \(\tr(R) = 2\cos(\theta/2)\).   If \(T\) is the matrix corresponding to a translation along a geodesic which passes through \(p\) then \(\tr(T) = 2\cosh(\dist(p,T(p))/2)\). 

Let \(o\) be the fixed point of \(A_{r_t}R\).  The triangle with vertices \(p,o,Ap\) has angles \(\pi/4,\pi/4,2\pi/t\) so that by the second law of cosines one has
\[a = \dist(p,o) = \acosh\left(\frac{1+\cos(2\pi/t)}{\sin(2\pi/t)}\right).\]

Notice that, because \(t = p/q\), for some integer \(k\) one has that \(X = (A_{r_t}R)^k\) is a rotation of angle \(2\pi/p\) fixing \(o\).

We will apply Jørgensen's inequality to \(X\) and \(R^2\).  For this purpose notice first that
\[|\tr(X)^2 -4| = |4\cos(\pi/p)^2 - 4| = 4(1-\cos(\pi/p)^2) = 4\sin(\pi/p)^2.\]

We now notice that 
\[[X,R^2] = XR^2X^{-1}R^{-2} = (XR^2X^{-1})R^2\]
is the composition of a central symmetry (i.e. a \(180º\) rotation) centered at \(p\), and a central symmetry centered at \(X(p)\).   It follows that \([X,R^2]\) is a translation of distance \(2\dist(p,X(p))\) along the geodesic passing through \(p\) and \(X(p)\).  

This implies that \(\tr([X,R^2]) = 2\cosh(\dist(p,X(p)))\).   Since \(p\) and \(X(p)\) are at distance \(a\) from \(o\) and the segments \([p,o]\) and \([X(p),o]\) form an angle of \(2\pi/p\) by the law of cosines one has
\begin{align*}
\tr([X,R^2]) - 2&= 2 (\cosh(\dist(p,X(p))) - 1)
\\ &= 2(\cosh(a)^2- \sinh(a)^2 \cos(2\pi/p) - 1)
\\ &= 2(1-\cos(2\pi/p))\sinh(a)^2
\\ &= 2(1-\cos(\pi/p)^2 + \sin(\pi/p)^2)\sinh(a)^2
\\ &= 4\sin(\pi/p)^2\sinh(a)^2.
\end{align*}

From this we obtain
\begin{align*}
|\tr(X)^2 - 4| + |\tr([X,R^2])- 2| &= 4\sin(\pi/p)^2\cosh(a)^2 
\\ &= 4\sin\left(\frac{\pi}{p}\right)^2\left(\frac{1+\cos\left(\frac{2\pi q}{p}\right)}{\sin\left(\frac{2\pi q}{p}\right)}\right)^2. 
\end{align*}

Denote the right hand side above by \(f(p,q)\), and notice that if \(q \ge 2\) and \(\frac{p}{q} > 4\) then \(f(p,q) \le f(p,2)\). 
So it suffices to show that \(f(p,2) < 1\) for all integers \(p\) with \(p/2 > 4\) (so \(p \ge 9\)).

Hence, setting \(x = \pi/p\) it suffices to show that
\[4\sin(x)^2\frac{(1+\cos(4x))^2}{\sin(4x)^2} < 1,\]
when \(0 < x < \pi/9\).  We will show that, in fact, the above inequality holds when \(x \in (0,\pi/4)\).

Reordering and taking square roots, we must prove that 
\[2(1+\cos(4x)) < \frac{\sin(4x)}{\sin(x)},\]
which applying the double angle formulas is equivalent to
\[2(1+\cos(2x)^2-\sin(2x)^2) = 4\cos(2x)^2 < \frac{2\sin(2x)\cos(2x)}{\sin(x)}.\]

For \(x \in (0,\pi/4)\) one has that \(\cos(2x)\) is positive, so the above is equivalent to
\[2\cos(2x) < \frac{\sin(2x)}{\sin(x)},\]
which using the double angle formula for \(\sin(2x)\) yields
\[2\cos(2x) < \frac{2\sin(x)\cos(x)}{\sin(x)} = 2\cos(x),\]
which holds for all \(x \in (0,\pi/4)\).
\end{proof}

\section{Proof of theorem \ref{dimensionthreetheorem}\label{threedimensionalsection}}

Let \(p,v_1,v_2,v_3,A_r,R_{12},R_{23},R_{31}\) be as defined preceeding the statement of theorem \ref{dimensionthreetheorem}.

If \(r \notin \lbrace r_n\rbrace \cup [r_\infty,+\infty)\) then considering the subgroups generated by \(A_r,R_{12}\) and \(A_r,R_{21}\) respectively and applying theorem \ref{maintheorem} one has that \(G_r\) is dense in the set of isometries of two perpendicular geodesically embedded copies of \(\H^2\) in \(\H^3\).  

It follows that, given any point \(q \in \H^3\), the closure of \(G_r\) contains the rotations fixing \(q\) with axis perpendicular to the two aforementioned hyperbolic planes.   This implies that the closure of \(G_r\) contains all rotations fixing \(q\) and therefore that \(G_r\) is dense in \(\isom(\H^3)\) in all these cases.

It remains to show that \(G_r\) is discrete for all \(r \in \lbrace r_5,r_6,\ldots\rbrace \cup [r_\infty,+\infty)\).

The same ping-pong argument given in the proof of Proposition \ref{trees} above (using six instead of four regions) yields that \(G_r\) is discrete if \(r \ge r_\infty\).
\begin{lemma}
 If \(r \ge r_\infty\) then \(G_r\) is discrete and preserves an embedding of the regular tree of degree six.
\end{lemma}
\begin{proof}
 Let \(T_1 = A_r,  T_2 = R_{12}A_r R_{12}^{-1}, T_3 = R_{31}A_r R_{31}^{-1},\) and for \(i = 1,2,3\) let
\(N_i = \lbrace q:  \dist(q,p) \ge \dist(q,T_i(p))\rbrace,\) and \(S_i = \lbrace q:  \dist(q,p) \ge \dist(q,T_i^{-1}(p))\rbrace\).

We claim that the six regions \(N_1,S_1,N_2,S_2,N_3,S_3\) are disjoint if and only if \(r \ge r_\infty\).

To establish the claim assume, without loss of generality (since one may permute and take the inverse of the transformations \(T_i\)), that \(N_1 \cap N_2 \neq \emptyset\) and let \(q' \in N_1 \cap N_2\).

Let \(q\) be the orthogonal projection of \(q'\) onto the plane \(P\) containing \(p\) and tangent vectors \(v_1,v_2\) (defined preceding the statement of theorem \ref{dimensionthreetheorem}).   The triangles with vertices \((p,q,q'), (T_1(p),q,q')\) and \((T_2(p),q,q')\) have a right angle at \(q\), and share the side joining \(q\) and \(q'\).   Since \(\dist(q',T_i(p)) \le \dist(q',p)\) for \(i = 1,2\) it follows that \(\dist(q,T_i(p)) \le \dist(q,p)\) as well.  Therefore \(q \in N_1 \cap N_2\).

Since the group generated by \(A_r\) and \(R_{12}\) preserves \(P\), it follows from proposition \ref{trees} that \(N_1 \cap P\) and \(N_2 \cap P\) are disjoint if and only if \(r \ge r_\infty\).    Hence \(N_1\) and \(N_2\) are disjoint if and only if \(r \ge r_\infty\) as claimed. 

Notice that \(T_i(\H^3 \setminus S_i) \subset N_i\) and \(T_i^{-1}(\H^3 \setminus N_i) \subset S_i\) for all \(i\).

Letting \(C = \H^3 \setminus \bigcup\limits_{i = 1}^3 (N_i \cup S_i)\) this implies that if \(X\) is any non-trivial reduced word in \(T_1, T_1^{-1}, T_2, T_2^{-1},T_3, T_3^{-1}\) then \(X(C) \cap C = \emptyset\).  Hence the group \(H_r\) generated by \(T_1,T_2,T_3\) is free and discrete.

We now claim that \(H_r\) has finite index in \(G_r\) and therefore \(G_r\) is also discrete.

To see this let \(S\) be the group generated by \(R_{12},R_{23}\), and \(R_{31}\).   Notice that \(S\) is finite and in fact \(|S| = 24\).

One has that \(G_r\) is generated by \(H_r\) and \(S\), and \(RH_rR^{-1} = H_r\) for all \(R \in S\) (it suffices to check this for the generators).  Hence, \(H_r\) is a normal subgroup of \(G_r\).
 
Furthermore, since \(R(C) = C\) for all \(R \in S\) it follows that if \(X,Y \in H_r\) and \(XR_1 = YR_2\) with \(R_1,R_2 \in S\) then \(X(C) = Y(C)\) and therefore \(X = Y\).  
This shows that \(H_rR_1\) and \(H_rR_2\) are disjoint, and it follows that \([G_r:H_r] = |S| = 24\) as claimed.

Notice furthermore that \(g(C) = C\) for \(g \in G_r\) if and only if \(g \in S\).

To conclude we now show that the action of \(G_r\) preserves a tree of degree six.

For this purpose for each \(i\) let \(n_i,s_i\) be the closest points to \(p\) in the regions \(N_i\) and \(S_i\) respectively.  The set \(\lbrace p\rbrace \cup \bigcup\limits_{i = 1}^3\lbrace n_i,s_i\rbrace\) is \(S\)-invariant and the union of geodesic segments 
\(A = \bigcup\limits_{i = 1}^3[p,s_i]\cup [p,n_i]\) intersect pairwise only at \(p\).   
Notice that \(X(A) \cap A \neq \emptyset\) for a non-identity element \(X \in H_r\) if and only if \(X \in \lbrace T_1,T_1^{-1},T_2,T_2^{-1},T_3,T_3^{-1}\rbrace\) in which case the intersection is a single point from the set \(\lbrace n_1,s_1,n_2,s_2,n_3,s_3\rbrace\).  It follows that the \(H_r\)-orbit of \(A\) is a tree of degree six (the Cayley graph of \(H_r\)) with all edges of length \(r\), and is invariant under \(G_r\) as claimed.
\end{proof}

We will now discuss the cases where \(r = r_n\) for \(n=5,6,\ldots\).

\begin{lemma}\label{dimthreelemma}
For all \(n \ge 5\) the group \(G_{r_n}\) is discrete.    The group \(G_{r_5}\) is cocompact, \(G_{r_6}\) is not cocompact but has finite covolume, and \(G_{r_n}\) has infinite covolume for all \(n \ge 7\).
\end{lemma}

To prove the result we will construct polyhedral tilings of \(\H^3\) which are preserved in each case.  Only in the case \(r = r_5\) are the polyhedra compact.

A finite sided polyhedron with sides which are regular \(n\)-gons with interior right-angles, and all dihedral angles equal to \(90º\), cannot exist if \(n \ge 6\).   To see this we give an argument communicated to us by Roland Roeder.

Suppose such a polyhedron exists for some \(n\), let \(V,E,F\) be the number of vertices, edges, and faces respectively.  Because the dihedral angles are non-obtuse each vertex is the intersection of exactly three faces by \cite[Proposition 1.1]{roeder-hubbard-dunbar}, so  \(V = nF/3\).   Since each edge is the intersection of two faces one has \(E = nF/2\).   Substituting this into Euler's polyhedron formula we obtain
\[V - E + F = \left(\frac{n}{3} - \frac{n}{2} + 1\right)F = \frac{6-n}{6}F = 2.\]

It follows that \(n \le 5\) from which \(n = 5\) is the only possibility in \(\H^3\).  We now show that this possibility actually occurs.

\begin{lemma}
There exists a convex hyperbolic dodecahedra \(C\) whose faces are regular right-angled hyperbolic pentagons.
\end{lemma}
\begin{proof}
By Andreev's theorem \cite[Proposition 1.1]{roeder-hubbard-dunbar}, there exists up to isometry a unique hyperbolic dodecahedron \(C\) such that the angle between any two faces at a shared edge is \(90º\).

It follows (for example from \cite[Proposition 1.1]{roeder-hubbard-dunbar}) that all the interior angles of each face are also right angles.   Hence, all faces are regular pentagons with interior right angles and their side length is \(r_5\).
\end{proof}

We will now show that, if \(n =6,7,\ldots\),  gluing hyperbolic \(n\)-gons at a right angle along each edge one bounds an infinite volume convex polyhedra in \(\H^3\).
\begin{lemma}\label{polyhedralemma}
For each \(n \ge 6\) there exists an infinite volume convex polyhedra in \(\H^3\) whose faces are regular \(n\)-gons with interior right angles (contained in a totally geodesic embedded hyperbolic plane), any two intersecting faces share a side and intersect at a right dihedral angle along this side, and exactly three sides meet at each vertex.
\end{lemma}
\begin{proof}
We will prove the case \(r = r_6\) separately.  See figure \ref{hthreefigure} for this case.

Consider the upper half space model of hyperbolic space.  On the boundary, take a tiling by regular (Euclidean) hexagons such that the distance between the centers of neighboring hexagons is \(2\).   
At the center of each hexagon consider a Euclidean sphere of radius \(\sqrt{2}\).
The intersection of each sphere with the upper half space is a geodesically embedded copy of \(\H^2\).
Furthermore, the copies corresponding to neighboring hexagons intersect at a right angle.  The part of each half sphere which is not contained in any other is a regular hyperbolic hexagon with interior right angles and exactly three of these meet at each vertex.

We define \(C\) as the region bounded by the constructed hexagons which does not accumulate on the boundary plane in this model.   Since \(C\) contains a horoball it has infinite volume.  In this case, and only for \(n = 6\), the polyhedra \(C\) has a single limit point on the geometric boundary of \(\H^3\) (the point corresponding to \(\infty\) in the upper half space model).

Suppose now that \(r = r_n\) for some natural number \(n \ge 7\).   Since \(\frac{1}{n} + \frac{1}{3} < \frac{1}{2}\) there exists a tiling of \(\H^2\) by regular (i.e. all sides and interior angles are equal) \(n\)-gons with exactly three meeting at each vertex.

Consider a totally geodesic embedding \(H\) of \(\H^2\) into \(\H^3\).   Tile \(H\) as described above.  Let \(x,y \in H\) be neighboring vertices in the tiling and consider unit speed geodesics \(\alpha,\beta\) perpendicular to \(H\) at \(\alpha(0) = x\) and \(\beta(0) = y\) respectively.  Assume furthermore that \(\alpha(t)\) and \(\beta(t)\) are on the same side of \(H\) for all \(t\).

Let \(s_n\) be the length of the side of the hyperbolic regular \(n\)-gon with interior angles of \(2\pi/3\).   Direct calculation shows that \(s_n < r_n\).

The distance between \(\alpha(t)\) and \(\beta(t)\) has minimum \(s_n\) at \(t = 0\) and goes to infinity when \(t \to +\infty\).  Therefore, there exists \(t_0 > 0\) such that this distance is exactly \(r_n\).

For each pair of vertices as above let \(\alpha(t_0),\beta(t_0)\) be vertices of the polyhedron to be constructed, and the geodesic segment between them be a side.   The geodesics \(\alpha,\beta\) are chosen so their positive direction is always the same fixed component of the complement of \(H\) in \(\H^3\).

Notice that the vertices and sides constructed from the tiling on \(H\) are equivariant under the group of isometries of \(\H^3\) which preserve the tiling and preserve each connected component of the complement of \(H\).  

In particular, considering the order \(n\) rotation along the geodesic perpendicular to \(H\) at the center of a tile in \(H\), one sees that the vertices constructed from those of the given tile are in a totally geodesically embedded hyperbolic plane in \(\H^3\) which is perpendicular to the axis of this rotation.

Hence, we may define a face of the polyhedron by considering this plane, and we have shown that it is a hyperbolic regular \(n\)-gon with interior right angles.

Since exactly three faces meet at each vertex, and the interior angles of all faces are right angles, it follows that the dihedral angle between faces sharing an edge is also a right angle.

Define \(C\) as the component bounded by these faces which contains \(H\). Since \(C\) contains a half space it has infinite volume.
\end{proof}

Using the polyhedra of lemma \ref{polyhedralemma} the discreteness of \(G_r\) follows from the Poincaré theorem for reflexion groups \cite[Chapter 3]{delaharpe} and a simple algebraic argument.
\begin{proof}[Proof of lemma \ref{dimthreelemma}]
Fix \(n\) and let \(C\) be the polyhedron given by lemma \ref{polyhedralemma}.

We may assume that the initial orthonormal frame is placed at a vertex of \(C\) and that the unit vectors in the frame point in the direction of the incident sides meeting at this vertex.

Let \(S\) be the group generated by the set of reflexions with respect to the faces of \(C\).   
By the Poincaré Polyhedron Theorem for reflexion groups \cite[Chapter 3]{delaharpe} the group \(S\) is discrete and \(C\) is a fundamental domain of its action.

On the other hand the group \(G\) of isometries of \(\H^3\) which stabilizes \(C\) is also discrete because the distance between distinct faces of \(C\) is bounded from below.

Observe that \(gSg^{-1} = S\) for all \(g \in G\) so that the group generated by \(G\) and \(S\) coincides with \(SG\), the set of elements of the form \(sg\) for some \(s \in S\) and \(g \in G\).

We claim that the group \(SG\) is discrete.

To see this suppose that \(s_kg_k\) is a sequence of elements in this group converging to the identity.   
One has \(s_k(g_k(C)) = s_k(C)\) and since \(C\) is fundamental domain for \(S\) it follows that \(s_k\) is the identity for all \(k\) large enough.   
However since \(G\) is discrete it follows that \(g_k\) is also the identity for all \(k\) large enough.  Hence, \(SG\) is discrete as claimed.

To conclude it suffices to show that \(G_r \subset SG\).

Recall that the initial point \(p\) is a vertex of \(C\) and the starting orthonormal frame vectors \(v_1,v_2,v_3\) point in the direction of the sides containing \(p\).   
Hence if \(R\) is any of the rotations \(R_{12},R_{23},R_{31}\) one has that \(R(C)\) shares a common face with \(C\).   Hence choosing \(s \in S\) to be the symmetry along that face one has \(sR \in G\) and therefore \(R \in SG\).

Similarly, because \(v_2\) and \(v_3\) belong to totally geodesically embedded hyperbolic planes containing the side in direction \(v_1\), 
one has that \(A_r(v_2)\) and \(A_r(v_3)\) are in the direction of two the sides containing \(A_r(p)\) other than the geodesic segment \([p,A_r(p)]\).  Hence, \(A_r(C)\) shares the face containing those two directions with \(C\).   Once again taking \(s\) to be the reflexion along this face one obtains \(A_r \in SG\).

This concludes the proof that \(G_r \subset SG\) from which it follows that \(G_r\) is discrete.

We will now discuss the covolume of \(G_{r_n}\) for \(n \ge 6\).

Letting \(n = 6\) notice from lemma \ref{polyhedralemma} that there is a unique boundary point \(\xi\) which is an accumulation point of \(C\).    Considering the union \(U\) of all half-geodesics starting at a face of \(C\) and ending at \(\xi\) notice that \(U\) must contain a fundamental domain of the action of \(G_{r_6}\).  Since \(U\) has finite volume it follows that \(G_{r_6}\) has finite covolume.

Now suppose that \(n \ge 7\), we claim that the quotient of \(C\) by its stabilizer has infinite volume.  This implies that claim that \(G_{r_7}\) has infinite covolume.

To establish the claim notice that the stabilizer of \(C\) coincides with that of the tiling of of the hyperbolic plane \(H\) considered in lemma \ref{polyhedralemma}.  Since one of the half spaces delimited by \(H\) is entirely contained in \(C\) the claim follows from the fact that any Fuchsian group acting on \(H\) has infinite covolume in \(\H^3\).  This, in turn, follows from the fact that the set \(U\) of half-geodesics perpendicular to \(H\) which start in a fundamental domain of the action on \(H\) has infinite volume.
\end{proof}

\begin{figure}[H]
\includegraphics[width = \textwidth]{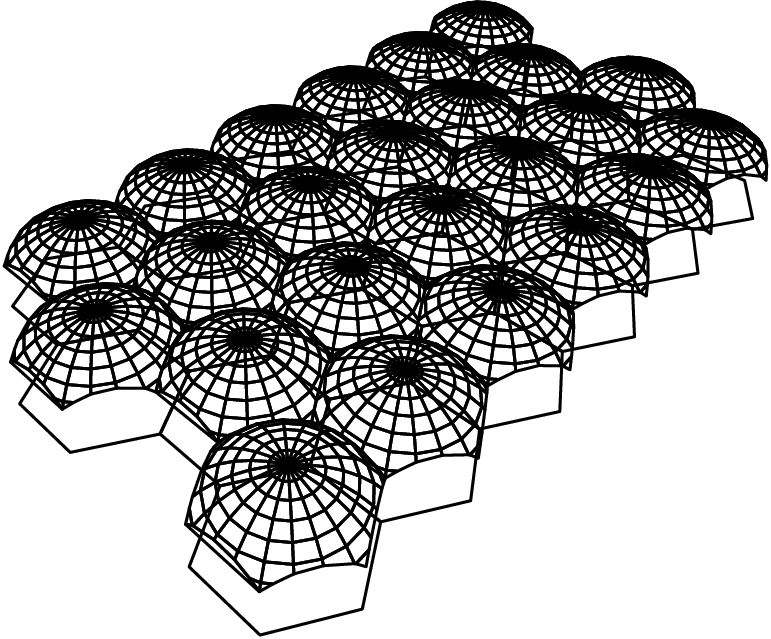}
\caption{Illustration of the proof of lemma \ref{polyhedralemma} for \(r = r_6\).  The spheres are centered at the midpoints of hexagons which tile a horizontal plane and their radii is such that two neighboring spheres intersect at a right angle.  In the upper half plane model, the surfaces obtained by truncating these spheres along planes perpendicular to the hexagonal tiling are hyperbolic right angled hexagons.\label{hthreefigure}}
\end{figure}

\section{Further discussion}

\subsection{Other primitive rotations\label{sectiongruet}}

Following \cite{gruet}, fix a natural number \(N \ge 2\), a real number \(s \in (0,1)\), and setting \(\xi = \exp(i\pi/N)\) let 
\[T_l(z) = \frac{z + s\xi^l}{s\overline{\xi}^l z + 1},\]
for \(0 \le l \le 2N - 1\).

Let \(G_{N,s}\) be the group of automorphisms of the unit disk \(\D = \lbrace z \in \C:|z| < 1\rbrace\) generated by \(T_0,T_1,\ldots, T_{2N-1}\).

Endowing \(\D\) with the hyperbolic metric each \(T_l\) is a translations of distance \(r = \log(\frac{1+s}{1-s})\).    The axis of translation for \(T_l\) and \(T_{l+1}\) intersect at \(0\) with an angle of \(\pi/N\). Hence, setting \(N = 2\), the group \(G_{2,s}\) is the same as \(H_{r}\) defined in section \ref{sectiontrees}.

In \cite[Theorem 2, part (i)]{gruet}, citing \cite[Theorem 3, part (i)]{cohen-colonna} for proof, it is claimed that if \(s \le \cos(\pi/2N)\) then the group \(G_{N,s}\) is not discrete.

Setting \(N = 2\), this would imply that \(H_r\) (which we recall is the group generated by \(A_r\) and \(RA_rR^{-1}\)) is not discrete for all 
\[r \le \log\left(\frac{1 + \frac{1}{\sqrt{2}}}{1 - \frac{1}{\sqrt{2}}}\right) = \log\left(\frac{\sqrt{2}+1}{\sqrt{2}-1}\right) = \log(3 + 2\sqrt{2}) = \acosh(3) = r_\infty,\]
contradicting the cases \(G_{r_{n}}, n \ge 5\) of theorem \ref{maintheorem}.

The mistake in the proof of \cite[Theorem 3, part (i)]{cohen-colonna} is that \cite[Lemma 2 and Lemma 3]{cohen-colonna} only show that the mapping \(\Phi\) from the \(2N\)-regular tree to \(\D\) considered by the authors is not an embedding.  But this does not entail that the group \(G_{N,s}\) is not discrete.

Despite this mistake the following question still seems natural and interesting:
\begin{question}
Let \(R\) be a rotation of even order \(2N\) fixing a point \(p \in \H^2\) and \(A_r\) a translation of distance \(r > 0\) along a geodesic containg \(p\).   For which values of \(r > 0\) is the group \(G_r\) generated by \(R\) and \(A_r\) discrete?
\end{question}

It seems that the methods used in the present article are sufficient to prove that there exists an increasing bounded sequence \(S\) such that \(G_r\) is discrete if and only if \(r \in S \cup [\sup S, +\infty)\).   However, a complete characterization of the sequence \(S\) does not follow immediately.

\subsection{Relationship to the Gilman-Maskit algorithm\label{sectiongilman}}

\begin{figure}[h]
\includegraphics[width = \textwidth]{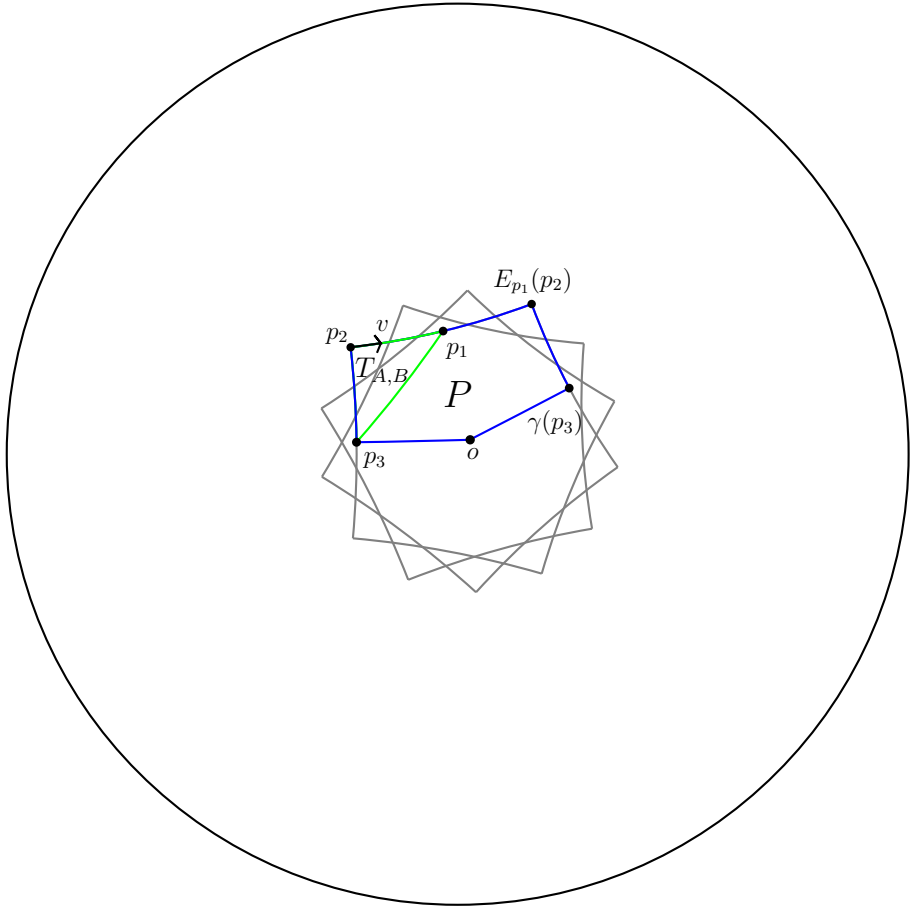}
\caption{Illustration corresponding to the Gilman-Maskit algorithm argument applied to the group \(H_{r}\) for \(r = r_{14/3}\).\label{figuregilman}}
\end{figure}

We fix in this section \(r = r_{14/3}\) and consider the group \(H_r\) generated by \(A_r\) and \(B_r\) as defined in section \ref{sectiontrees}.   By Theorem \ref{maintheorem} the group \(G_r\) is not discrete (see figure \ref{figuregilmannondiscrete}, and therefore \(H_r\) is not discrete (since it is a finite index subgroup of \(G_r\)).

The key step of our proof of non-discreteness of \(H_r\) is the application of Jørgensen's inequality to suitable elements of \(G_r\) (see Proposition \ref{nonprimitiveproposition}).  

We will now apply the arguments of \cite{gilman} (which are much more general since they cover all groups generated by two translations with intersecting axes) to the generators \(A_r\) and \(B_r\) of \(H_r\) for the sake of comparison.  To keep with Gilman's notation set \(A = A_r, B = B_r\) and \(G = H_r\).  See figure \ref{figuregilman} where the following discussion is illustrated.

We start with the Discreteness Theorem \cite[Theorem 3.1.1]{gilman}.   

In our particular case the commutator \([A,B]\) is a rotation of angle \(4 \times 2\pi \frac{3}{14}\). 
Hence, we land in case 4 of the theorem with \(\tr([A,B]) = -2\cos\left(k 2\pi/n\right)\) with \(k = 3\) and \(n = 7\).

The discussion is given in terms of an acute triangle \(\text{Act}_{A,B}\).   To calculate this triangle one begins setting \(p = p_2\) and \(p_1 = A_{r/2}(p)\) and \(p_3 = B_{r/2}(p)\) and considers the triangle \(T_{A,B}\) with vertices \(p_1,p_2,p_3\).  In our case \(T_{A,B}\) is a right isosceles triangle, and therefore the algorithm given in \cite[Section 2]{gilman} stops immediately and \(T_{A,B} = \text{Act}_{A,B}\).

This implies that, since \(k = 3\), and \(\text{Act}_{A,B}\) is a right isosceles triangle, the group is discrete according to \cite[Theorem 3.1.1]{gilman} contradicting theorem \ref{maintheorem}.

However \cite[Theorem 3.2.1]{gilman} states that if \(\text{Act}_{A,B}\) is a right isosceles triangle then one must have \(k = 2\) which is also a contradiction.

It seems that the mistake is only in the statement of the results and not the proofs. Going further into the arguments of \cite[Section 13]{gilman} one sees that the the key point of the argument is the Matelski-Beardon count stated in \cite[Theorem A.0.2]{gilman}.   

According to the Matelski-Beardon count we should consider the group \(G^*\) generated by the central symmetries (rotations of angle \(180º\)) \(E_{p_1},E_{p_2},E_{p_3}\) centered at \(p_1,p_2,p_3\) respectively.  
Let \(P\) be the pentagon with vertices \(p_3,p_2,E_{p_1}(p_2)\), \(\gamma(p_3)\), and \(o\) where \(\gamma =E_{p_1}E_{p_2}E_{p_3}\) and \(o\) is the fixed point of \(\gamma\).   

Assuming that \(G*\) is discrete let \(t\) be the quotient between the area of \(P\) and the area of \(\H^2/G^*\).   The theorem implies that if \(k = 3\) then \(t = 2\).

Hence, verifying that one cannot have \(t = 2\) yields an alternative proof of non-discreteness of  \(G\) from the one given above.

\begin{figure}
\begin{minipage}{\textwidth}
\begin{minipage}{0.32\textwidth}
\includegraphics[width=\textwidth]{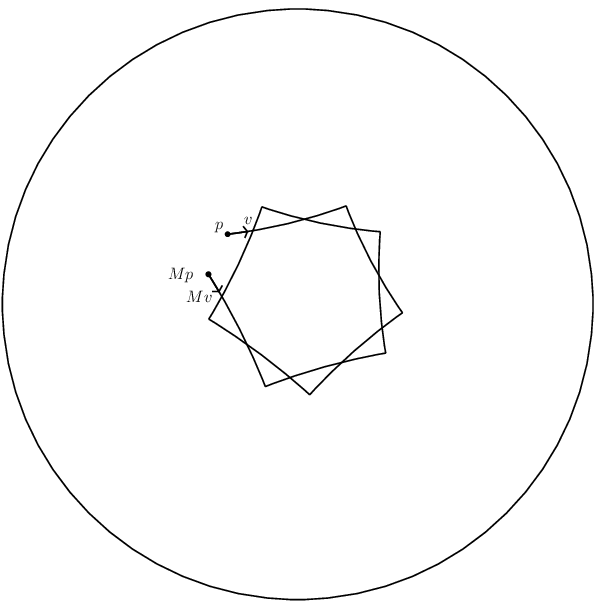}
\end{minipage}
\begin{minipage}{0.32\textwidth}
\includegraphics[width=\textwidth]{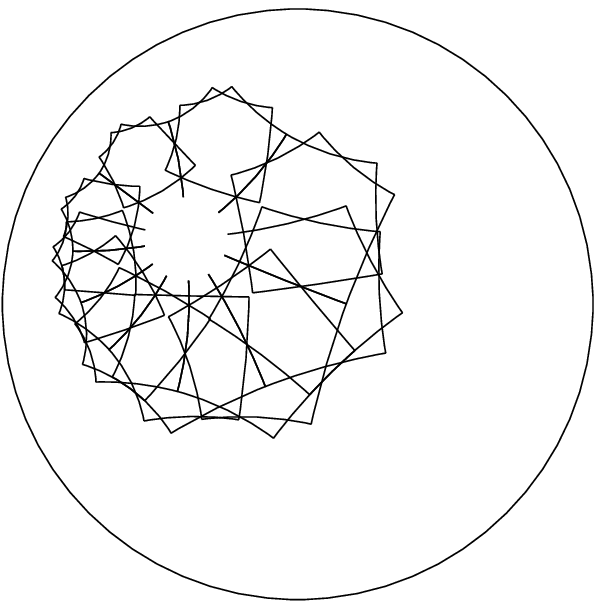}
\end{minipage}
\begin{minipage}{0.32\textwidth}
\includegraphics[width=\textwidth]{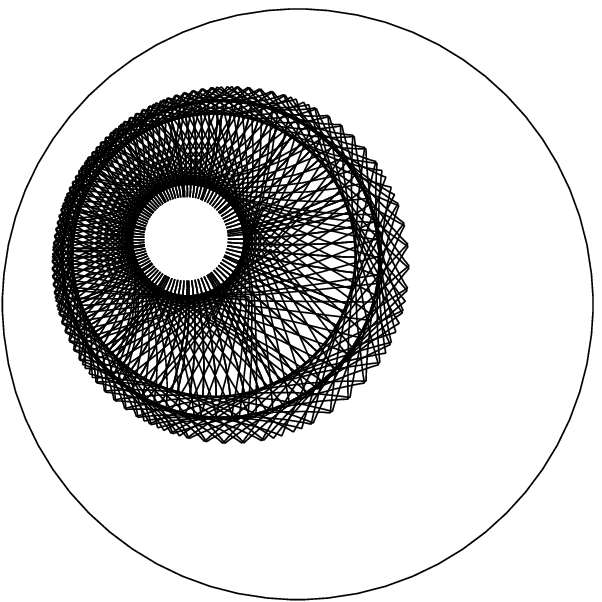}
\end{minipage}
\caption{An illustration of \(1, 10\) and \(100\) iterations of \(M = (A_rR)^9 R\) for \(r = r_{14/3}\).\label{figuregilmannondiscrete}}
\end{minipage}
\end{figure}


\begin{thebibliography}{SCMT91}

\bibitem[ABY10]{avila-bochi-yoccoz}
Artur Avila, Jairo Bochi, and Jean-Christophe Yoccoz.
\newblock Uniformly hyperbolic finite-valued {${\rm SL}(2,\mathbb{
  R})$}-cocycles.
\newblock {\em Comment. Math. Helv.}, 85(4):813--884, 2010.

\bibitem[Ad81]{abelson-disessa}
Harold Abelson and Andrea~A. diSessa.
\newblock {\em Turtle geometry}.
\newblock MIT Press, Cambridge, Mass.-London, 1981.
\newblock The computer as a medium for exploring mathematics, MIT Press Series
  in Artificial Intelligence.

\bibitem[Arm88]{armstrong}
M.~A. Armstrong.
\newblock {\em Groups and symmetry}.
\newblock Undergraduate Texts in Mathematics. Springer-Verlag, New York, 1988.

\bibitem[Bea95]{beardon}
Alan~F. Beardon.
\newblock {\em The geometry of discrete groups}, volume~91 of {\em Graduate
  Texts in Mathematics}.
\newblock Springer-Verlag, New York, 1995.
\newblock Corrected reprint of the 1983 original.

\bibitem[BGS85]{ballmann-gromov-schroeder}
Werner Ballmann, Mikhael Gromov, and Viktor Schroeder.
\newblock {\em Manifolds of nonpositive curvature}, volume~61 of {\em Progress
  in Mathematics}.
\newblock Birkh\"{a}user Boston, Inc., Boston, MA, 1985.

\bibitem[CC94]{cohen-colonna}
Joel~M. Cohen and Flavia Colonna.
\newblock Embeddings of trees in the hyperbolic disk.
\newblock {\em Complex Variables Theory Appl.}, 24(3-4):311--335, 1994.

\bibitem[dlH91]{delaharpe}
Pierre de~la Harpe.
\newblock An invitation to {C}oxeter groups.
\newblock In {\em Group theory from a geometrical viewpoint ({T}rieste, 1990)},
  pages 193--253. World Sci. Publ., River Edge, NJ, 1991.

\bibitem[dR71]{derham}
G.~de~Rham.
\newblock Sur les polygones g\'{e}n\'{e}rateurs de groupes fuchsiens.
\newblock {\em Enseignement Math.}, 17:49--61, 1971.

\bibitem[EP94]{epstein}
David B.~A. Epstein and Carlo Petronio.
\newblock An exposition of {P}oincar\'{e}'s polyhedron theorem.
\newblock {\em Enseign. Math. (2)}, 40(1-2):113--170, 1994.

\bibitem[Gil95]{gilman}
Jane Gilman.
\newblock Two-generator discrete subgroups of {${\rm PSL}(2,{\bf R})$}.
\newblock {\em Mem. Amer. Math. Soc.}, 117(561):x+204, 1995.

\bibitem[Gil19]{gilman2}
Jane Gilman.
\newblock Roots and rational powers in {$PSL(2,\Bbb R)$} and discreteness.
\newblock {\em Geom. Dedicata}, 201:139--154, 2019.

\bibitem[GM91]{gilman-maskit}
J.~Gilman and B.~Maskit.
\newblock An algorithm for {$2$}-generator {F}uchsian groups.
\newblock {\em Michigan Math. J.}, 38(1):13--32, 1991.

\bibitem[Gru08]{gruet}
Jean-Claude Gruet.
\newblock Hyperbolic random walks.
\newblock In {\em S\'{e}minaire de probabilit\'{e}s {XLI}}, volume 1934 of {\em
  Lecture Notes in Math.}, pages 279--294. Springer, Berlin, 2008.

\bibitem[Jø76]{jorgensen}
Troels Jørgensen.
\newblock On discrete groups of {M}\"{o}bius transformations.
\newblock {\em Amer. J. Math.}, 98(3):739--749, 1976.

\bibitem[Kap16]{kapovich}
Michael Kapovich.
\newblock Discreteness is undecidable.
\newblock {\em Internat. J. Algebra Comput.}, 26(3):467--472, 2016.

\bibitem[Kob12]{koberda2012}
Thomas Koberda.
\newblock Ping-pong lemmas with applications to geometry and topology.
\newblock In {\em Geometry, topology and dynamics of character varieties},
  volume~23 of {\em Lect. Notes Ser. Inst. Math. Sci. Natl. Univ. Singap.},
  pages 139--158. World Sci. Publ., Hackensack, NJ, 2012.

\bibitem[Led13]{ledrappier}
Fran\c{c}ois Ledrappier.
\newblock Mesures stationnaires sur les espaces homog\`enes (d'apr\`es {Y}ves
  {B}enoist et {J}ean-{F}ran\c{c}ois {Q}uint).
\newblock {\em Ast\'{e}risque}, (352):Exp. No. 1058, x--xi, 535--556, 2013.
\newblock S\'{e}minaire Bourbaki. Vol. 2011/2012. Expos\'{e}s 1043--1058.

\bibitem[Lesa]{lessa2}
Pablo Lessa.
\newblock Dibujos hiperbólicos.
\newblock Github.
\newblock https://github.com/pablolessa/dibujos\_hiperbolicos.

\bibitem[Lesb]{lessa}
Pablo Lessa.
\newblock Hyperbolic turtle graphics.
\newblock Github.
\newblock https://github.com/pablolessa/hturtle.

\bibitem[Mas71]{maskit}
Bernard Maskit.
\newblock On {P}oincar\'{e}'s theorem for fundamental polygons.
\newblock {\em Advances in Math.}, 7:219--230, 1971.

\bibitem[RHD07]{roeder-hubbard-dunbar}
Roland K.~W. Roeder, John~H. Hubbard, and William~D. Dunbar.
\newblock Andreev's theorem on hyperbolic polyhedra.
\newblock {\em Ann. Inst. Fourier (Grenoble)}, 57(3):825--882, 2007.

\bibitem[SCMT91]{simscoomber-martin-thorne}
Helen Sims-Coomber, Ralph Martin, and Michael Thorne.
\newblock A non-euclidean implementation of logo.
\newblock {\em Computers \& Graphics}, 15(1):117 -- 130, 1991.

\bibitem[Sul85]{sullivan}
Dennis Sullivan.
\newblock Quasiconformal homeomorphisms and dynamics. {II}. {S}tructural
  stability implies hyperbolicity for {K}leinian groups.
\newblock {\em Acta Math.}, 155(3-4):243--260, 1985.

\end{thebibliography}
\end{document}